\RequirePackage{fix-cm}
\documentclass[smallextended]{svjour3}       
\smartqed  
\usepackage{amsmath,amsfonts,amssymb,mathtools}
\usepackage{graphicx,subfigure}
\usepackage{braket,dsfont,url,enumerate}

\usepackage{pdfpages} 
\usepackage{cmap}
\usepackage{hyperref}
\usepackage{color,dsfont,braket}
\usepackage{bbm}
\usepackage[headings]{fullpage}
\usepackage{fullpage}
\usepackage{fancyhdr}
\usepackage{times}
\usepackage{cmap}
\usepackage{pgf}
\usepackage{comment}

\newcommand{\C}{\mathbb{C}}

\newcommand{\na}{\nabla}
\newcommand{\pa}{\partial}
\newcommand{\eps}{\varepsilon}

\newcommand{\Om}{\Omega}

\newcommand{\IOm}{I\times \Om}

\newcommand{\Id}{\operatorname{Id}}

\newcommand{\norm}[1]{\lVert#1\rVert}

\newcommand{\abs}[1]{\lvert#1\rvert}
\newcommand{\Ppol}[1]{\mathcal{P}_{#1}}

\newcommand{\sgn}{\operatorname{sgn}}

\newcommand{\lh}{\abs{\ln{h}}}
\newcommand{\lk}{\ln{\frac{T}{k}}}

\newcommand{\Xk}{X^0_k}
\newcommand{\Xkh}{X^{q,r}_{k,h}}
\newcommand{\vertiii}[1]{{\left\vert\kern-0.25ex\left\vert\kern-0.25ex\left\vert #1
    \right\vert\kern-0.25ex\right\vert\kern-0.25ex\right\vert}}

\definecolor{darkred}{rgb}{.7,0,0}

\definecolor{green}{rgb}{0,0.7,0}

\begin{document}

\title{Discrete maximal parabolic regularity  for  Galerkin finite element methods}


\author{Dmitriy Leykekhman         \and
        Boris Vexler 
}


\institute{D.~Leykekhman \at
              Department of Mathematics,
               University of Connecticut,
              Storrs,
              CT~06269, USA \\
              \email{leykekhman@math.uconn.edu}           
           \and
           B.~Vexler \at
              Lehrstuhl f\"ur Optimalsteuerung, Technische Universit\"at M\"unchen,
Fakult\"at f\"ur Mathematik, Boltzmannstra{\ss}e 3, 85748 Garching b. M\"unchen, Germany\\
\email{vexler@ma.tum.de}
}

\date{Received: date / Accepted: date}

\maketitle

\begin{abstract}
The main goal of the paper is to establish time semidiscrete  and space-time fully discrete maximal parabolic regularity for the time discontinuous Galerkin solution  of linear parabolic equations. Such estimates have many applications. They are essential, for example, in establishing optimal a priori error estimates in non-Hilbertian norms  without unnatural coupling of spatial mesh sizes with time steps.
\end{abstract}

\section{Introduction}
Let $\Om$ be a Lipschitz domain in $\mathbb{R}^d$, $d=2,3$ and $I=(0,T)$.
We consider the heat equation as a model of a parabolic second order partial differential equation,
\begin{equation}\label{eq: heat equation}
\begin{aligned}
\pa_tu(t,x)-\Delta u(t,x) &= f(t,x), & (t,x) &\in \IOm,\;  \\
    u(t,x) &= 0,    & (t,x) &\in I\times\pa\Omega, \\
   u(0,x) &= u_0(x),    & x &\in \Omega
\end{aligned}
\end{equation}
with a right-hand side $f \in L^s(I;L^p(\Omega))$ for some $1\le p,s\le \infty$ and $u_0\in L^p(\Om)$,  $1\le p\le \infty$.

The maximal parabolic regularity for $u_0\equiv 0$ says that there exists a constant $C$ such that,
\[
\norm{\pa_t u}_{L^s(I;L^p(\Omega))} + \norm{\Delta u }_{L^s(I;L^p(\Omega))} \le C \, \norm{f}_{L^s(I;L^p(\Omega))}, \quad 1<p,s<\infty,\quad \text{for all}\quad f\in L^s(I;L^p(\Omega)),
\]
(see, e.g.,~\cite{CoulhonT_DuongXT_2000,Haller-DintelmannR_RehbergJ_2009,HieberM_PrussJ_1997}).
The maximal parabolic regularity is an important analytical tool and has a number of applications, especially to nonlinear problems and/or optimal control problems when sharp regularity results are required (cf.~\cite{HombergD_MeyerC_RehbergJ_RingW_2009,KrumbiegelK_RehbergJ_2013a,PieperK_KunischK_VexlerB_2014,LeykekhmanD_VexlerB_2013a}).
Our aim in this paper is to establish similar maximal parabolic regularity results for time discrete discontinuous Galerkin solutions as well as for the fully discrete Galerkin approximations. Such results are very useful, for example, in  fully discrete a  priori error estimates  and are essential in order to keep the spatial mesh size $h$ and the time steps $k$ independent of each other (cf.~\cite{LeykekhmanD_VexlerB_2015a}). In~\cite{LeykekhmanD_VexlerB_2015d} we apply the results of this paper to establish pointwise best approximation estimates for fully discrete Galerkin solutions.

Maximal parabolic regularity with applications to semidiscrete finite element Galerkin solutions in space were analyzed for smooth domains in \cite{MGeissert_2006a,MGeissert_2007a} and for convex polyhedra in \cite{LiB_2015a}. Time discrete results are much less known in the finite element community. Explicit methods are treated in \cite{BlunckS_2001a,BlunckS_2001b,PortalP_2005}. Implicit Euler methods with pointwise norms in time are considered in \cite{GuidettiD_2007b,GuidettiD_2007a}. A more systematic investigation of discrete maximal parabolic regularity for various time schemes was carried out by  Sobolevski\u{i} and Ashyralyev and  summarized in the book \cite{AshyralyevA_SobolevskiuiPE_1994}.

In this paper, we investigate maximal parabolic regularity for a family of time discontinuous Galerkin (dG) methods, which were first deeply analyzed for linear second order parabolic problems in \cite{ErikssonK_JohnsonC_ThomeeV_1985}.
There is a number of important properties that make the dG schemes attractive for temporal discretization of parabolic problems. For example,  such schemes allow for a priori error estimates of optimal order with respect to discretization parameters, such as the size of time steps and the mesh width, as well as with respect to the regularity requirements for the solution (see, e.g.,~\cite{ErikssonK_JohnsonC_1991a,ErikssonK_JohnsonC_1995a}). Different systematic approaches for a posteriori error estimation and adaptivity developed for finite element discretizations can be adapted for dG temporal discretization of parabolic equations, (see, e.g.,~\cite{SchmichM_VexlerB_2008,SchotzauD_WihlerTP_2010}).   Since the trial space allows for discontinuities at the time nodes, the use of different spatial discretizations for each time step can be directly incorporated into the discrete formulation,  (see, e.g.,~\cite{SchmichM_VexlerB_2008}). Compared to the continuous Galerkin methods, dG schemes are not only A-stable but also strongly A-stable, (see, e.g.,~\cite{LasaintP_RaviartPA_1974}). An efficient and easy to implement approach that avoids complex coefficients, which arise in the equations obtained by a direct decoupling for high order dG schemes, was developed in \cite{RichterT_SpringerA_VexlerB_2013}. For the treatment of optimal control problems, Galerkin methods are particularly suitable since they expose an important property that the two approaches optimize-then-discretize, i.e., the discretization of the optimality system built up on the continuous level, and discretize-then-optimize,
i.e., discretization of the state equation and subsequent construction of the optimality system on the discrete level, lead to the same discretization scheme, (see, e.g.,~\cite{BeckerR_MeidnerD_VexlerB_2007a}). Compared to continuous Petrov-Galerkin time-stepping schemes
(see \cite{DMeidner_BVexler_2011a} for details), dG schemes also have the advantage that the adjoint state can use the same discretization as the state variable. This allows for unified numerical treatment and simplifies a priori and a posteriori error analysis, (see, e.g.,~\cite{ChrysafinosK_2007,DMeidner_BVexler_2007a,DMeidner_BVexler_2008a,DMeidner_BVexler_2008b}).

The main results of this paper for the time semidiscrete  discontinuous Galerkin $u_k$ solution consist roughly of two parts. First, for the homogeneous problem (i.e. $f=0$) with $u_0 \in L^p(\Omega)$, $1\le p\le \infty$ we show
\begin{equation}\label{eq: smoothing dgr intro combined}
\|\pa_tu_{k} \|_{L^\infty(I_m; L^p(\Om))}+\|\Delta u_{k}\|_{L^\infty(I_m; L^p(\Om))}+\left\|
\frac{[u_{k}]_{m-1}}{k_m}\right\|_{L^p(\Om)}\le \frac{C}{t_m}\|u_0\|_{L^p(\Om)},
\end{equation}
for $m=1,2,\dots,M$. Then, using this smoothing result, we also establish discrete maximal parabolic regularity for  the inhomogeneous problem when $u_0=0$. We show,
\begin{equation}\label{eq: maximal parabolic dgr intro combined}
\left(\sum_{m=1}^M\|\pa_t u_{k}\|^s_{L^s(I_m;L^p(\Om))}\right)^{\frac{1}{s}}+\|\Delta u_{k}\|_{L^s(I;L^p(\Om))}+\left(\sum_{m=1}^Mk_m\left\|\frac{[u_{k}]_{m-1}}{k_m} \right\|^s_{L^p(\Om)}\right)^{\frac{1}{s}}\le C\lk\|f\|_{L^s(I;L^p(\Om))},
\end{equation}
for  $1\le s\le \infty$ and  $1\le p\le \infty$, with obvious notation changes in the case of $s=\infty.$ In the case of the lowest order piecewise constant method, i.e., $q=0$, the first terms on the left-hand side of the above estimates vanish. In contrast to the continuous case, the limiting cases $s,p \in \{1,\infty\}$ are allowed, which explains the logarithmic factor in~\eqref{eq: maximal parabolic dgr intro combined}.
 We also provide the fully discrete analog of \eqref{eq: smoothing dgr intro combined} and \eqref{eq: maximal parabolic dgr intro combined}.

The rest of the paper is organized as follows. In the next section we introduce the discretization method and the resolvent estimates, which build the main analytical tool of the paper. For better communication of  the ideas we first analyze the dG($0$) method, which is technically much simpler, and in the following section we analyze the general dG($q$) case. That is done  in Sections \ref{sec: dG0} and \ref{sec: dGq}, respectively. At the end of Section \ref{sec: dGq} we provide an example of how such maximal parabolic regularity results can rather easily lead to optimal order error estimates. Finally, Section \ref{sec: fully discrete} is devoted to fully discrete Galerkin solutions. In Section~\ref{sec:general_norms} we provide an extension of our results to the case of a general norm fulfilling a resolvent estimate. This generalization, being of an independent interest, is used, for example, in~\cite{LeykekhmanD_VexlerB_2015d} for derivation of pointwise interior (local) error estimates of fully discrete Galerkin solutions.

\section{Preliminaries} \label{sec: preliminaries}
To introduce the time discontinuous Galerkin discretization for the problem,
 we partition  $I =(0,T)$ into subintervals $I_m = (t_{m-1}, t_m]$ of length $k_m = t_m-t_{m-1}$, where $0 = t_0 < t_1 <\cdots < t_{M-1} < t_M =T$. The maximal and minimal time steps are denoted by $k =\max_{m} k_m$ and $k_{\min}=\min_{m} k_m$, respectively.
We impose the following conditions on the time mesh (as in ~\cite{DMeidner_RRannacher_BVexler_2011a}):
\begin{enumerate}[(i)]
  \item There are constants $c,\beta>0$ independent on $k$ such that
    \[
      k_{\min}\ge ck^\beta.
    \]
  \item There is a constant $\kappa>0$ independent on $k$ such that for all $m=1,2,\dots,M-1$
    \[
    \kappa^{-1}\le\frac{k_m}{k_{m+1}}\le \kappa.
    \]
  \item It holds $k\le\frac{1}{4}T$.
\end{enumerate}
The semidiscrete space $X_k^q$ of piecewise polynomial functions in time is defined by
\[
X_k^q=\{u_{k}\in L^2(I;H^1_0(\Om)) :\ u_{k}|_{I_m}\in \Ppol{q}(H^1_0(\Om)), \ m=1,2,\dots,M \},
\]
where $\Ppol{q}(V)$ is the space of polynomial functions of degree $q$ in time with values in a Banach space $V$.
We will employ the following notation for functions in $X_k^q$
\begin{equation}\label{def: time jumps}
u^+_m=\lim_{\eps\to 0^+}u(t_m+\eps), \quad u^-_m=\lim_{\eps\to 0^+}u(t_m-\eps), \quad [u]_m=u^+_m-u^-_m.
\end{equation}
Next we define the following bilinear form
\begin{equation}\label{eq: bilinear form B}
 B(u,\varphi)=\sum_{m=1}^M \langle \partial_t u,\varphi \rangle_{I_m \times \Omega} + (\na u,\na \varphi)_{\IOm}+\sum_{m=2}^M([u]_{m-1},\varphi_{m-1}^+)_\Om+(u_{0}^+,\varphi_{0}^+)_\Om,
\end{equation}
where $(\cdot,\cdot)_{\Omega}$ and $(\cdot,\cdot)_{I_m \times \Omega}$ are the usual $L^2$ space and space-time inner-products,
$\langle \cdot,\cdot \rangle_{I_m \times \Omega}$ is the duality product between $ L^2(I_m;H^{-1}(\Omega))$ and $ L^2(I_m;H^{1}_0(\Omega))$. We note, that the first sum vanishes for $u \in \Xk$. The dG($q$) semidiscrete (in time) approximation $u_k\in X_k^q$ of \eqref{eq: heat equation} is defined as
\begin{equation}\label{eq: semidiscrete heat with RHS}
B(u_k,\varphi_k)=(f,\varphi_k)_{\IOm}+(u_0,\varphi_{k,0}^+)_\Om \quad \text{for all }\; \varphi_k\in X_k^q.
\end{equation}
Rearranging the terms in \eqref{eq: bilinear form B}, we obtain an equivalent (dual) expression of $B$:
\begin{equation}\label{eq:B_dual}
 B(u,\varphi)= - \sum_{m=1}^M \langle u,\partial_t \varphi \rangle_{I_m \times \Omega} + (\na u,\na \varphi)_{\IOm}-\sum_{m=1}^{M-1} (u_m^-,[\varphi]_m)_\Om + (u_M^-,\varphi_M^-)_\Om.
\end{equation}

The analysis of such schemes in non-Hilbertian setting is usually done by using a semigroup approach that represents time stepping formulas as a Dunford-Taylor integral in the complex plane \cite[~Ch. 9]{ThomeeV_2006}. This approach requires certain resolvent estimates.
For Lipschitz domains and a given $\gamma\in (0,\pi/2)$, the resolvent estimate (see \cite{ShenZW_1995a})  guarantees the existence of a constant $C$ such that for all $u\in L^p(\Om)$, $1\le p\le \infty$, and any $z \in \C \setminus \Sigma_{\gamma}$ the following estimate holds:
\begin{equation}\label{eq: continuous resolvent}
\|(z+\Delta)^{-1}u\|_{L^p(\Om)}\le  \frac{C}{1+\abs{z}}\|u\|_{L^p(\Om)},
\end{equation}
where the Laplace operator $-\Delta$ is supplemented with homogeneous Dirichlet boundary conditions, and
\begin{equation}\label{eq: definition of sigma}
{\Sigma_{\gamma}}= \{z \in \mathbb{C} : \abs{\arg{(z)}} \le \gamma\}.
\end{equation}
Using the identity $\Delta(z+\Delta)^{-1}=\Id-z(z+\Delta)^{-1}$, one immediately obtains,
\begin{equation}\label{eq: continuous resolvent with laplace}
\|\Delta(z+\Delta)^{-1}u\|_{L^p(\Om)}\le  C\|u\|_{L^p(\Om)},\quad  z \in \C \setminus\Sigma_{\gamma}, \quad 1\le p\le \infty,\quad u\in L^p(\Om).
\end{equation}
We note, that all our results for semidiscrete solutions hold if we replace the Laplace operator $-\Delta$ with a more general self-adjoint second order elliptic operator $A$ provided it satisfies \eqref{eq: continuous resolvent}.


\section{Estimates for dG($0$)} \label{sec: dG0}

For the ease of the presentation, we first establish the results for the lowest order piecewise constant discretization dG($0$). In this case, we use the following notation,
\begin{equation}\label{eq: dG0 + -}
u_{k,m}=u_k|_{I_m}, \quad u_{k,m}^+=u_{k,m+1},\quad u_{k,m}^-=u_{k,m},\quad m=1,2,\dots,M-1.
\end{equation}
First, we establish results for the homogeneous problem. In this case the dG($0$) method is equivalent to the Backward Euler method.
\subsection{Results for the homogeneous problem}
Let $f=0$, $u_0\in L^p(\Om)$ and let $u_k\in X_k^0$ be the semidiscrete approximation  of \eqref{eq: heat equation} defined by
\begin{equation}\label{eq: homogeneous dg parabolic}
B(u_k,\chi_k)=(u_{0},\chi_{k,1}), \quad \forall \chi_k\in X^0_k,
\end{equation}
i.e., the dG($0$) solution $u_k$ satisfies
\begin{equation}\label{eq: homogeneous dg0 parabolic one step}
\begin{aligned}
u_{k,1}-k_1\Delta u_{k,1}&=u_0,\\
u_{k,m}-k_m\Delta u_{k,m}&=u_{k,m-1}, \quad m=2,3,\dots,M.
\end{aligned}
\end{equation}
The first result shows that the solution can not grow from one time step to the next one.
\begin{lemma}\label{lemma: monotonicity}
Let $u_k$ be the solution of \eqref{eq: homogeneous dg parabolic}. Then, for $u_0 \in L^p(\Om)$, $1\le p\le \infty$ there holds
\[
\|u_{k,m}\|_{L^p(\Om)} \le \|u_0\|_{L^p(\Om)}  \quad \forall m = 1,2,\dots,M.
\]
\end{lemma}
\begin{proof}
First, we assume $u_0 \in L^\infty(\Omega)$ and establish
\begin{equation}\label{eq:Linfty maximamum}
\|u_{k,m}\|_{L^\infty(\Om)} \le \|u_0\|_{L^\infty(\Om)}  \quad m = 1,2,\dots,M.
\end{equation}
It is sufficient  to consider only a single time step,
\begin{equation}
u_{k,1}-k_1 \Delta u_{k,1} = u_{0}.
\end{equation}
We want to show that $\|u_{k,1}\|_{L^\infty(\Om)}\le \|u_{0}\|_{L^\infty(\Om)}$. Assume it is false. Let $x_0\in \Om$ be a point where $u_{k,1}$ attains a maximum. By~\cite[Theorem 3.3]{Haller-DintelmannR_MeyerC_RehbergJ_SchielaA_2009}, we know that $u_{k,1}\in C_0(\Om)$, hence, there exists an open ball $B_\delta(x_0)$ of radius $\delta>0$ centered at $x_0$ with $\overline{B_\delta(x_0)} \subset \Omega$ such that
$$
u_{k,1}(x)>\|u_0\|_{L^\infty(\Om)} \quad \text{for all } \; x\in B_\delta(x_0).
$$
Hence,
$$
u_{k,1}(x)-u_{0}(x)>0 \quad \text{on}\ B_\delta(x_0).
$$
By the maximum principle, from
$$
-\Delta u_{k,1} = \frac{1}{k_1}\left(u_{0}-u_{k,1}\right)<0 \quad \text{on}\ B_\delta(x_0),
$$
we obtain a contradiction to the assumption that $u_{k,1}$ has a maximum at the interior point $x_0$. This contradiction  establishes \eqref{eq:Linfty maximamum}.
Next, using a duality argument, we will show
\begin{equation}\label{eq:L1 maximamum}
\|u_{k,1}\|_{L^1(\Om)}\le \|u_0\|_{L^1(\Om)}.
\end{equation}
Consider the problem, to find $z_{k,1}\in H^1_0(\Om)$ that satisfies,
$$
z_{k,1}-k_1\Delta z_{k,1} =z_0, \quad \text{with $z_0=\sgn u_{k,1}$}.
$$
The solution $z_{k,1}$ can be thought of as a single step of the dG($0$) method to a parabolic problem with initial condition $\sgn u_{k,1}$.
Thus,
$$
\|u_{k,1}\|_{L^1(\Om)}=(u_{k,1},z_{0})=(z_{k,1},u_{k,1})+k_1(\na z_{k,1}, \na u_{k,1})=(u_{0},z_{k,1})\le \|u_{0}\|_{L^1(\Om)}\|z_{0}\|_{L^\infty(\Om)}\le \|u_{0}\|_{L^1(\Om)},
$$
where  we have used \eqref{eq:Linfty maximamum} for $z_k$ and the fact that $\|z_{0}\|_{L^\infty(\Om)}=\norm{\sgn{u}_{k,1}}_{L^\infty(\Om)}=1$.
This establishes \eqref{eq:L1 maximamum}. Interpolating, we obtain the lemma for $1\le p\le \infty$.
\end{proof}

Next we will establish a smoothing result.
\begin{theorem}[Homogeneous smoothing estimate]\label{thm: homogeneous laplace}
Let $u_k\in X_k^0$ be the solution of \eqref{eq: homogeneous dg parabolic} with $u_0 \in L^p(\Omega)$, $ 1\le p\le \infty$. Then there exists a constant $C$ independent of $k$ such that
$$
\|\Delta u_{k,m}\|_{L^p(\Om)}\le \frac{C}{t_m}\|u_0\|_{L^p(\Om)},  \quad m = 1,2,\dots,M.
$$
\end{theorem}
\begin{proof}
The proof is given on page 1321 in \cite{ErikssonK_JohnsonC_LarssonS_1998a} for the $L^2(\Om)$ norm, but the proof is valid for the $L^p(\Om)$ norm as well by using the resolvent estimate \eqref{eq: continuous resolvent} with respect to the $L^p(\Om)$ norm.
\end{proof}

\begin{remark}\label{remark: homogenous laplaca}
Let $u_k\in X_k^0$ be the solution of \eqref{eq: homogeneous dg parabolic} with $u_0 \in L^p(\Omega)$, $ 1\le p\le \infty$. Then there exists a constant $C$ independent of $k$ such that
$$
\|u_{k,m}\|_{L^p(\Om)} + (t_m-t_l)\|\Delta u_{k,m}\|_{L^p(\Om)}\le C\|u_{k,l}\|_{L^p(\Om)},\quad \forall m>l\ge 1.
$$
\end{remark}
From \eqref{eq: homogeneous dg0 parabolic one step}, we immediately obtain the following result.
\begin{corollary}\label{cor: homogeneous jumps}
Let $u_k\in X_k^0$ be the solution of \eqref{eq: homogeneous dg parabolic} with $u_0 \in L^p(\Omega)$, $ 1\le p\le \infty$. Then there exists a constant $C$ independent of $k$ such that
$$
\left\|
\frac{[u_{k}]_{m-1}}{k_m}\right\|_{L^p(\Om)}\le \frac{C}{t_m}\|u_0\|_{L^p(\Om)},  \quad m = 1,2,\dots,M.
$$
\end{corollary}

\subsection{Results for the inhomogeneous problem}

Now we consider $u_{k}\in X_k^0$ to be the dG($0$) solution to the parabolic equation with $u_0=0$, i.e., $u_k$ satisfies,
\begin{equation}\label{eq: dG nonhomogeneous equation}
B(u_k,\varphi_k)=(f,\varphi_k)_{I\times \Om},\quad \forall \varphi_k\in X_k^0.
\end{equation}
Thus, the dG($0$) solution satisfies
\begin{equation}\label{eq: one step dG0 inhomogeneous}
\begin{aligned}
u_{k,1}-k_1\Delta u_ {k,1} &= k_1f_1,\\
u_{k,m}-k_m\Delta u_ {k,m} &= u_{k,m-1}+k_mf_m,\quad m=2,3,\dots,M,
\end{aligned}
\end{equation}
where
$$
f_m(\cdot)=\frac{1}{k_m}\int_{I_m}f(t,\cdot)dt.
$$
Since $f_m$  is  the $L^2$ projection  of $f$ onto the piecewise constant functions on each subinterval $I_m$, we have
\begin{subequations} \label{eq: estimate for fm in lp}
\begin{align}
\max_{1\le m\le M}\|f_m\|_{L^p(\Om)}&\le C\|f\|_{L^\infty(I;L^p(\Om))}, \quad 1\le p\le\infty,\\
 \sum_{m=1}^M k_m\|f_m\|^r_{L^p(\Om)}&\le C\|f\|^r_{L^r(I;L^p(\Om))},\quad 1\le p\le\infty, \quad 1\le r<\infty.
\end{align}
\end{subequations}

We now state our main result for the dG($0$) approximations.
\begin{theorem}[Maximal parabolic regularity]\label{thm: maximal parabolic regularity}
Let $1\le s,p\le \infty$ and $u_0=0$. Then, there exists a constant $C$ independent of $k$ such that for every
 $f\in L^s(I;L^p(\Om))$ and $u_k$ satisfying \eqref{eq: dG nonhomogeneous equation}, the following estimate holds:
$$
\|\Delta u_k\|_{L^s(I;L^p(\Om))}\le C\lk\| f\|_{L^s(I;L^p(\Om))}, \quad 1\le s\le \infty, \quad 1\le p\le \infty.
$$
\end{theorem}
\begin{proof}
Using \eqref{eq: one step dG0 inhomogeneous}, we can write the dG($0$) solution as
$$
u_{k,m}=\sum_{l=1}^m k_l\left(\prod_{j=1}^{m-l+1}r(-k_{m-j+1}\Delta)\right)f_l,\quad m=1,2,\dots,M,
$$
where $r(z)=(1+z)^{-1}.$
Then,
$$
\Delta u_{k,m}=\sum_{l=1}^m k_l\left(\Delta\prod_{j=1}^{m-l+1} r(-k_{m-j+1}\Delta)\right)f_l,\quad m=1,2,\dots,M.
$$
Hence
$$
\|\Delta u_{k,m}\|_{L^p(\Om)}\le\sum_{l=1}^m k_l\left\|\left(\Delta\prod_{j=1}^{m-l+1} r(-k_{m-j+1}\Delta)\right)f_l\right\|_{L^p(\Om)},\quad m=1,2,\dots,M.
$$
From Remark \ref{remark: homogenous laplaca}, since each term in the sum on the right-hand side can be thought of as a homogeneous solution with initial condition $f_l$ at $t=t_{l-1}$, we have
$$
\left\|\left(\Delta\prod_{j=1}^{m-l+1} r(-k_{m-j+1}\Delta)\right)f_l\right\|_{L^p(\Om)}\le \frac{C}{t_m-t_{l-1}}\|f_l\|_{L^p(\Om)}.
$$
Thus, we obtain
\begin{equation}\label{eq: before Holder}
\|\Delta u_{k,m}\|_{L^p(\Om)}\le C\sum_{l=1}^m  \frac{k_l}{t_m-t_{l-1}}\|f_l\|_{L^p(\Om)},\quad m=1,2,\dots,M.
\end{equation}
For $s=\infty$, we obtain from the above estimate and using \eqref{eq: estimate for fm in lp},
\begin{equation*}
\begin{aligned}
\|\Delta u_{k}\|_{L^\infty(I;L^p(\Om))}&=\max_{1\le m\le M}\|\Delta u_{k,m}\|_{L^p(\Om)}\le C\max_{1\le m\le M}\sum_{l=1}^m \frac{k_l}{t_m-t_{l-1}}\|f_l\|_{L^p(\Om)}\\
&\le C\max_{1\le l\le M}\|f_l\|_{L^p(\Om)}\max_{1\le m\le M}\sum_{l=1}^m  \frac{k_l}{t_m-t_{l-1}}\le C\lk \|f\|_{L^\infty(I;L^p(\Om))},
\end{aligned}
\end{equation*}
where in the last step we used that
\begin{equation}\label{eq: estimating sum by integral for log}
\sum_{l=1}^m  \frac{k_l}{t_m-t_{l-1}}\le 1+\int_0^{t_{m-1}}\frac{dt}{t_m-t}=1+\ln{\frac{t_m}{k_m}}\le C\lk,
\end{equation}
by using the assumption $k_{\min}\geq Ck^\beta$ and $k\le \frac{T}{4}$.

For $s=1$, we have
$$
\|\Delta u_{k}\|_{L^1(I;L^p(\Om))}=\sum_{m=1}^Mk_m\|\Delta u_{k,m}\|_{L^p(\Om)}\le C\sum_{m=1}^Mk_m\sum_{l=1}^m \frac{k_l}{t_m-t_{l-1}}\|f_l\|_{L^p(\Om)}.
$$
Changing the order of summation and using \eqref{eq: estimate for fm in lp}, we obtain
$$
\begin{aligned}
\|\Delta u_{k}\|_{L^1(I;L^p(\Om))}&\le C\sum_{l=1}^M k_l\|f_l\|_{L^p(\Om)} \sum_{m=l}^M \frac{k_m}{t_m-t_{l-1}}\\&\le C\lk\sum_{l=1}^M k_l\|f_l\|_{L^p(\Om)}\le C\lk\|f\|_{L^1(I;L^p(\Om))},
\end{aligned}
$$
where we used again that
$$
\sum_{m=l}^M \frac{k_m}{t_m-t_{l-1}}\le C\lk.
$$
Interpolating between $s=1$ and $s=\infty$, we obtain the result for any $1\le s\le \infty$.
\end{proof}
\begin{remark}
The appearance of the logarithmic term is natural for the critical values $s=1$, $p=1$, $s=\infty$, or $p=\infty$, since the corresponding maximal parabolic regularity results for the continuous problem  hold only for $1<s,p<\infty$. For $s=2$ or $p=2$,
the power of the logarithm can be lowered. Thus, for  $p=2$, from \cite{DMeidner_BVexler_2008a} we know,
$$
\|\Delta u_{k}\|_{L^2(I;L^2(\Om))}\le C\|f\|_{L^2(I;L^2(\Om))}
$$
and from \eqref{eq: before Holder}, we have
$$
\begin{aligned}
\|\Delta u_{k}\|_{L^s(I;L^2(\Om))}\le C\lk\|f\|_{L^s(I;L^2(\Om))}, \quad 1\le s\le\infty.
\end{aligned}
$$
Interpolating between $s=2$ and $s=\infty$ and between $s=2$ and $s=1$, we obtain
$$
\|\Delta u_{k}\|_{L^s(I;L^2(\Om))}\le C\left(\lk\right)^{\frac{\abs{s-2}}{s}}\|f\|_{L^s(I;L^2(\Om))},  \quad \text{for any $1\le s\le\infty$}.
$$
Similarly, we can obtain,
$$
\|\Delta u_{k}\|_{L^2(I;L^p(\Om))}\le C\left(\lk\right)^{\frac{\abs{p-2}}{p}}\|f\|_{L^2(I;L^p(\Om))},  \quad \text{for any $1\le p\le\infty$}.
$$
\end{remark}

\begin{corollary}[Maximal parabolic regularity for jumps]\label{cor: maximal parabolic regularity jumps}
Let $1\le s,p\le \infty$ and $u_0=0$. Then, there exists a constant $C$ independent of $k$ such that for every
 $f\in L^s(I;L^p(\Om))$ and $u_k$ satisfying \eqref{eq: dG nonhomogeneous equation}, the following estimate holds,
\begin{align*}
\max_{1\le m \le M}\left\|\frac{[u_k]_{m-1}}{k_m}\right\|_{L^p(\Om)}&\le C\lk\| f\|_{L^\infty(I;L^p(\Om))}, \quad 1\le p\le \infty,\\
\left(\sum_{m=1}^Mk_m\left\|\frac{[u_k]_{m-1}}{k_m}\right\|^s_{L^p(\Om)}\right)^{\frac{1}{s}}&\le C\lk\| f\|_{L^s(I;L^p(\Om))}, \quad 1\le s< \infty, \quad 1\le p\le \infty,
\end{align*}
where the jump term $[u_k]_0$ at $t = 0$ is defined as $u_{k,1}$.
\end{corollary}
\begin{proof}
Since by \eqref{eq: one step dG0 inhomogeneous} on each time subinterval $I_m$ we have
$$
k_m^{-1}[u_k]_{m-1}=\Delta u_{k,m}+f_m, \quad m=1,2,\dots,M,
$$
by using Theorem \ref{thm: maximal parabolic regularity}, we have
$$
\max_{1\le m \le M}k_m^{-1}\left\|[u_k]_{m-1}\right\|_{L^p(\Om)}\le \max_{1\le m \le M}\left(\|\Delta u_{k,m}\|_{L^p(\Om)}+\|f_m\|_{L^p(\Om)}\right)\le C\lk\| f\|_{L^\infty(I;L^p(\Om))}.
$$
Similarly, using Theorem \ref{thm: maximal parabolic regularity}, for $1\le s<\infty$ we have
$$
\sum_{m=1}^Mk_m\left\|\frac{[u_k]_{m-1}}{k_m}\right\|^s_{L^p(\Om)}\le C_s\sum_{m=1}^Mk_m\left(\left\|\Delta u_{k,m}\right\|^s_{L^p(\Om)}+\left\| f_{m}\right\|^s_{L^p(\Om)}\right)\le C_s\left(\lk\right)^s\| f\|^s_{L^s(I;L^p(\Om))},
$$
where the constant $C_s$ depends on $s$.
By taking the $s$-root we obtain the corollary.
\end{proof}

\section{Estimates for dG($q$)} \label{sec: dGq}

In this section we will establish the dG($q$) version of the results from the previous section. It is convenient to introduce some additional notation. Let $q\geq 1$ and $\psi_l(t)\in P_q([0,1])$, $l=0,1,\dots,q$ be the standard Lagrange basis functions  on the interval $[0,1]$, i.e., $\psi_l\left(\frac{j}{q}\right)=\delta_{lj}$, where $\delta_{lj}$ is the Kronecker symbol. Then for any $u_k\in X_k^q$ on the time interval $I_m=(t_{m-1},t_m]$ we have
\begin{equation}\label{eq: formular of dGq of u_k on I_m}
u_k|_{I_m}=\sum_{l=0}^q U^m_{l}(x)\psi_l\left(\frac{t-t_{m-1}}{k_m}\right),
\end{equation}
with $U^m_l\in H^1_0(\Om)$ independent of $t$.
In this notation, we have
$$
u_{k,m}^+=U^{m+1}_{0}\quad\text{and}\quad u_{k,m}^-=U^{m}_q.
$$


\subsection{Results for the homogeneous problem}

Let $u_{k}\in X^q_k$ be the semidiscrete in time solution to the parabolic equation with $f\equiv 0$, i.e.,
\begin{equation}\label{eq: dGr homogeneous equation}
B(u_k,\varphi_k)=(u_0,\varphi^+_{k,0}),\quad \forall \varphi_k\in X_k^q.
\end{equation}
Alternatively, on a single interval $I_m$, we have
\begin{equation}\label{eq: one step dGr homogenesous}
\begin{aligned}
U^1_l&=r_{l,0}(-k_1\Delta )u_0,\quad l=0,1,\dots,q,\\
U^m_l&=r_{l,0}(-k_m\Delta )U^{m-1}_q,\quad l=0,1,\dots,q,\quad m=2,3,\dots,M,
\end{aligned}
\end{equation}
where the rational functions $r_{l,0}$ are  of the form,
\begin{equation}\label{eq: rational function rlj}
r_{l,0}(\lambda)=\frac{p_{l,0}(\lambda)}{\hat{p}(\lambda)},\quad l=0,1,\dots,q,
\end{equation}
with $\hat{p}$ being a polynomial of degree $q+1$ with no roots on the right-half complex plane  and $p_{l,0}$, $l=0,1,\dots,q$ being polynomials of degree $q$ (cf.~\cite{ErikssonK_JohnsonC_LarssonS_1998a}, page 1322). Since $r_{q,0}(\lambda)$ is  a subdiagonal Pad\'{e} approximation of $e^{-\lambda}$, we also have (cf.~\cite{EgertM_RozendaalJ_2013})
\begin{equation}\label{eq: proprties of r_00}
r_{q,0}(0)=p_{q,0}(0)=\hat{p}(0)=1\quad \text{and}\quad |r_{q,0}(\lambda)-e^{-\lambda} |=O(|\lambda|^{2q+2}),
\end{equation}
as $\lambda\to 0$. The rational functions $r_{l,0}$ satisfy  the following properties, which we will often use
\begin{equation}\label{eq: proprties of r_l0}
r_{l,0}(0)=1, \quad\text{and}\quad r_{l,0}(\lambda)-1=\frac{\lambda \tilde{p}_{l}(\lambda)}{\hat{p}(\lambda)}, \quad l=0,1,\dots,q,
\end{equation}
where $\tilde{p}_l(\lambda)$ are some polynomials of degree $q$.
The first property follows, for example, by considering the homogeneous Neumann problem with initial condition $u_0=1$. Then the exact solution $u$ and the dG($q$) solution $u_k$ are the same and equal to 1, i.e., $u=u_k=1$. Hence, all nodal values $U^m_l=1$ for all $m=1,2,\dots,M$ and $l=0,1,\dots,q$.  For example for $m=1$, we have
$$
1 = U^1_l = r_{l,0}(-k_1\Delta )u_0 = r_{l,0}(-k_1\Delta )1 = r_{l,0}(0),
$$
and as a result $r_{l,0}(0)=1$.
The second property in \eqref{eq: proprties of r_l0} is just a consequence of the first one.

\begin{remark}
The dG($1$) solution $u_k$ on each subinterval $I_m$ is of the form
$$
U^m_0\left(\frac{t_m-t}{k_m}\right)+U^m_1\left(\frac{t-t_{m-1}}{k_m}\right)
$$
and the rational functions are
$\hat{p}(\lambda)=1+\frac{2}{3}\lambda+\frac{\lambda^2}{6}$, $r_{0,0}(\lambda)=1+\frac{2}{3}\lambda$, and $r_{1,0}(\lambda)=1-\frac{\lambda}{3}$.
\end{remark}

For later proof we require two supplementary results.
\begin{lemma}\label{lem: raional function estimates}
Let the rational function $r(z)$ be of the form
$r(z)= \frac{p(z)}{\hat{p}(z)},$
where $\hat{p}(z)$ is a polynomial of degree $q+1$ with no roots on the right half complex plane and $p(z)$ is a polynomial of degree $q$, for some $q\geq 0$.
 Then, there exists a constant $C$ independent of $k>0$, such that for any $g\in L^p(\Om)$
\begin{equation}
\|r(-k\Delta) g\|_{L^p(\Om)}\le C\|g\|_{L^p(\Om)}.
\end{equation}
\end{lemma}
\begin{proof}
For simplicity we assume that the roots $z_1,z_2,\dots,z_q$ of $\hat{p}$ are pairwise distinct. If it is not the case, the argument can be slightly modified.
For $q=0$ we have $r(z) = \frac{c_0}{z-z_0}$ and the desired estimate follows directly by the resolvent estimate~\eqref{eq: continuous resolvent}, since
\[
r(-k\Delta)g = -\frac{c_0}{k} \left(\frac{z_0}{k}+\Delta\right)^{-1} g
\]
and therefore by~\eqref{eq: continuous resolvent}
\[
\|r(-k\Delta) g\|_{L^p(\Om)}\le \frac{\abs{c_0}}{k} \frac{C}{1+\frac{\abs{z_0}}{k}}\|g\|_{L^p(\Om)} \le \frac{C \abs{c_0}}{\abs{z_0}} \|g\|_{L^p(\Om)}.
\]
For $q>0$ we use the partial fraction decomposition
\[
r(z) = \sum_{i=0}^q \frac{c_i}{z-z_i}
\]
with some $c_i \in \mathbb{C}$. Applying the estimate for $q_0$ to each summand we obtain
\[
\|r(-k\Delta) g\|_{L^p(\Om)}\le C \left( \sum_{i=0}^q  \frac{\abs{c_i}}{\abs{z_i}}\right) \|g\|_{L^p(\Om)},
\]
which completes the proof.
\end{proof}
\begin{lemma}\label{lemma: rational with z}
Let the rational function $r(z)$ be of the form
$r(z)= \frac{zp(z)}{\hat{p}(z)}$,
where $\hat{p}(z)$ is a polynomial of degree $q+1$ with no roots on the right-half complex plane and $p(z)$ is a polynomial of degree $q$, for some $q\geq 0$.
Then for any $g\in L^p(\Om)$ with $\Delta g\in L^p(\Om)$, $1\le p\le \infty$, there exists a constant $C$ independent of $k$ such that
$$
\|r(-k\Delta)g\|_{L^p(\Om)} \le Ck\|\Delta g\|_{L^p(\Om)}.
$$
\end{lemma}
\begin{proof}
This lemma is just a consequence of the previous one. We set $\tilde r(z)=\frac{p(z)}{\hat{p}(z)}$ and obtain:
\[
r(-k\Delta)g = -k \Delta \, \tilde r(-k\Delta) g = -k \, \tilde r(-k\Delta) \Delta g.
\]
The the result follows by Lemma~\ref{lem: raional function estimates}.
\end{proof}
\begin{lemma}\label{lem: raional function estimates one more}
Let the rational function $r(z)$ be of the form
$r(z)= \frac{zp(z)}{\hat{p}(z)},$
where $\hat{p}(z)$ is a polynomial of degree $q+1$ with no roots on the right half complex plane and $p(z)$ is a polynomial of degree $q$, for some $q\geq 1$.
 Then, there exists a constant $C$ independent of $k$, such that for any $g\in L^p(\Om)$
\begin{equation}
\|r(-k\Delta) g\|_{L^p(\Om)}\le C\|g\|_{L^p(\Om)}.
\end{equation}
\end{lemma}
\begin{proof}
We set $\tilde r(z)=\frac{p(z)}{\hat{p}(z)}$ and obtain:
\[
\|r(-k\Delta) g\|_{L^p(\Om)} \le k\|\Delta \tilde r(-k\Delta) g\|_{L^p(\Om)}.
\]
The estimate
\[
\|\Delta \tilde r(-k\Delta) g\|_{L^p(\Om)} \le \frac{C}{k} \|g\|_{L^p(\Om)}
\]
is provided on the top of page 1322 in \cite{ErikssonK_JohnsonC_LarssonS_1998a} using a decomposition $r(z)=r_1(z)+r_2(z)$, where $r_1(z)=\frac{c}{z-z_0}$, with $z_0$ being a root of $\hat{p}(z)$ and $c$ such that the degree of the polynomial in the numerator of $r_2(z)$ is less or equal $q-1$. Then the estimate for $\Delta \tilde r_1(-k\Delta) g$ follows directly by applying a dG($0$) type argument and the term $\Delta \tilde r_2(-k\Delta) g$ is estimated using the Dunford-Taylor formula.
\end{proof}
Next we provide some properties of the dG($q$) solutions of the homogeneous problem.
\begin{lemma}\label{lemma: monotonicity dG_r interior}
Let $u_k$ be the solution of \eqref{eq: dGr homogeneous equation} with $u_0 \in L^p(\Omega)$, $1\le p\le \infty$.  Then,
$$
\|u_{k}\|_{L^\infty(I_m;L^p(\Om))}\le C\|u_{0}\|_{L^p(\Om)}, \quad \forall m = 1,2,\dots,M.
$$
\end{lemma}

\begin{proof}
The proof is given in \cite[~Thm. 5.1]{ErikssonK_JohnsonC_LarssonS_1998a} for the $L^2(\Om)$ norm, but the proof is valid for the $L^p(\Om)$ norm as well by using the resolvent estimate \eqref{eq: continuous resolvent} with respect to the $L^p(\Om)$ norm.
\end{proof}

\begin{theorem}[Homogeneous smoothing estimate]\label{thm: homogeneous smoothing dG_r}
Let $u_k$ be the solution of \eqref{eq: dGr homogeneous equation} with $u_0 \in L^p(\Omega)$, $1\le p\le \infty$. Then there exists a constant $C$ independent of $k$ such that
$$
\|\Delta u_k\|_{L^\infty(I_m; L^p(\Om))}\le \frac{C}{t_m}\|u_0\|_{L^p(\Om)}, \quad m=1,,2\dots,M.
$$
\end{theorem}
\begin{proof}
Again the proof is given in \cite[~Thm. 5.1]{ErikssonK_JohnsonC_LarssonS_1998a} for the $L^2(\Om)$ norm, but the proof is valid for the $L^p(\Om)$ norm as well by using the resolvent estimate \eqref{eq: continuous resolvent} with respect to the $L^p(\Om)$ norm.
\end{proof}
\begin{remark}\label{remark: homogoeneous smoothing termwise}
Notice that the statement of Theorem \ref{thm: homogeneous smoothing dG_r} is equivalent to
\begin{equation}
\|\Delta U_l^m\|_{L^p(\Om)}\le \frac{C}{t_m}\|u_0\|_{L^p(\Om)}, \quad  \quad m=1,2,\dots,M, \quad l=0,1,\dots, q,
\end{equation}
which we will use in the following proofs.
\end{remark}
\begin{remark}\label{remark: homogenous laplaca dG_r}
Let $u_k$ be the solution of \eqref{eq: dGr homogeneous equation}. Then there exists a constant $C$ independent of $k$ such that
$$
\|u^-_{k,m}\|_{L^p(\Om)}+(t_m-t_n)\|\Delta u_{k,m}\|_{L^\infty(I_m; L^p(\Om))}\le C\|u^-_{k,n}\|_{L^p(\Om)},\quad m>n,\quad  \quad n=1,2,\dots,M,
$$
or in terms of nodal values
\begin{equation}
\|U^{m}_{q}\|_{L^p(\Om)}+(t_m-t_n)\|\Delta U_l^m\|_{L^p(\Om)}\le C\|U^{n}_{q}\|_{L^p(\Om)},\quad m>n, \quad  \quad n=1,2,\dots,M, \quad l=0,1,\dots, q.
\end{equation}
\end{remark}

\begin{theorem}[Homogeneous smoothing estimate for jumps]\label{thm: homogeneous smoothing dG_r jumps}
Let $u_k$ be the solution of \eqref{eq: dGr homogeneous equation} with $u_0 \in L^p(\Omega)$, $1\le p\le \infty$. Then there exists a constant $C$ independent of $k$ such that
$$
\left\|\frac{[u_k]_{m-1}}{k_m} \right\|_{L^p(\Om)}\le \frac{C}{t_m}\|u_0\|_{L^p(\Om)}, \quad m=1,2,\dots,M,
$$
where $[u_k]_{0}=U_0^1-u_0$.
\end{theorem}
\begin{proof}
For $m>1$, using \eqref{eq: one step dGr homogenesous}, we have
$$
[u_k]_{m-1}=U^m_0-U^{m-1}_q=r_{0,0}(-k_m\Delta)U^{m-1}_q-U^{m-1}_q=(r_{0,0}(-k_m\Delta)-\operatorname{Id})U^{m-1}_q.
$$
Using \eqref{eq: proprties of r_l0} and Lemma \ref{lemma: rational with z},
we obtain
$$
\left\|\frac{[u_k]_{m-1}}{k_m} \right\|_{L^p(\Om)}\le C\left\|\Delta U^{m-1}_q \right\|_{L^p(\Om)}.
$$
Now by Remark \ref{remark: homogoeneous smoothing termwise} and the assumption on the time mesh $(ii)$, we obtain
$$
\left\|\Delta U^{m-1}_q \right\|_{L^p(\Om)}\le \frac{C}{t_{m-1}}\|u_0\|_{L^p(\Om)}\le  \frac{C}{t_{m}}\|u_0\|_{L^p(\Om)}.
$$
That finishes the proof for this case.

For $m=1$, by Lemma \ref{lemma: monotonicity dG_r interior} we have,
$$
\left\|\frac{[u_k]_{0}}{k_1} \right\|_{L^p(\Om)}=\frac{1}{k_1}\| U_0^1-u_0\|_{L^p(\Om)}\le \frac{C}{k_1}\| u_0\|_{L^p(\Om)}=\frac{C}{t_1}\| u_0\|_{L^p(\Om)}.
$$
\end{proof}

Similarly, we can obtain the corresponding result for the time derivative.
\begin{theorem}[Homogeneous smoothing estimate for time derivatives]\label{thm: homogeneous smoothing derivative dGr}
Let $u_k$ be the solution of \eqref{eq: dGr homogeneous equation} with $u_0 \in L^p(\Omega)$, $1\le p\le \infty$. Then there exists a constant $C$ independent of $k$ such that
$$
\|\pa_tu_k \|_{L^\infty(I_m; L^p(\Om))}\le \frac{C}{t_m}\|u_0\|_{L^p(\Om)}.
$$
\end{theorem}
\begin{proof}
For $m>1$,
using \eqref{eq: formular of dGq of u_k on I_m} and \eqref{eq: one step dGr homogenesous}, we have
$$
\pa_t u_k|_{I_m}=k_m^{-1}\sum_{l=0}^q U^m_{l}(x)\psi_l'\left(\frac{t-t_{m-1}}{k_m}\right)=k_m^{-1}\sum_{l=0}^q r_{l,0}(-k_m\Delta) \psi_l'\left(\frac{t-t_{m-1}}{k_m}\right)U^{m-1}_{q}(x).
$$
By the fact that  $\sum_{l=0}^q \psi_l\left(\frac{t-t_{m-1}}{k_m}\right)=1$ we have $\sum_{l=0}^q \psi_l'\left(\frac{t-t_{m-1}}{k_m}\right)=0$. Using~\eqref{eq: proprties of r_l0}, i.e., $r_{l,0}(0)=1$ we obtain
$$
\sum_{l=0}^q r_{l,0}(z) \psi_l'\left(\frac{t-t_{m-1}}{k_m}\right)=\frac{z\tilde{p}_t(z)}{\hat{p}(z)},
$$
where $\hat{p}(z)$ is the same polynomial as in \eqref{eq: rational function rlj} and $\tilde{p}_t(z)$ is some polynomial of degree $q-1$ whose coefficients are time dependent, but uniformly bounded on $I_m$.
Thus again by Lemma \ref{lemma: rational with z},
we obtain
$$
\|\pa_t u_k\|_{L^\infty(I_m;L^p(\Om))}\le C\|\Delta U^{m-1}_{q}\|_{L^p(\Om)}.
$$
Remark \ref{remark: homogoeneous smoothing termwise} and the assumption on the time mesh $(ii)$, finishes the proof for $m>1$.

For $m=1$, by Lemma \ref{lemma: monotonicity dG_r interior} we have,
$$
\|\pa_tu_k \|_{L^\infty(I_1; L^p(\Om))}\le Ck_1^{-1}\sum_{l=0}^q \|U^1_l\|_{L^p(\Om)}\|\psi_l'\|_{L^\infty(I_1)}\le \frac{C}{t_1}\|u_0\|_{L^p(\Om)}.
$$
\end{proof}

\subsection{Results for the inhomogeneous problem}
In this section we establish properties of the dG($q$) solution $u_{k}\in X^q_k$ to the inhomogeneous parabolic equation with $u_0=0$, that satisfies,
\begin{equation}\label{eq: dGr nonhomogeneous equation}
B(u_k,\varphi_k)=(f,\varphi_k),\quad \forall \varphi_k\in X_k^q.
\end{equation}
Alternatively, on a single time interval $I_m$, we have
\begin{equation}\label{eq: dG(r) inhomogeneous one step}
\begin{aligned}
U^1_l &= k_1\sum_{j=0}^q r_{l,j}(-k_1\Delta)f^1_j, \quad l=0,1,\dots,q,\\
U^m_l &= r_{l,0}(-k_m\Delta )U^{m-1}_q+k_m\sum_{j=0}^q r_{l,j}(-k_m\Delta)f^m_j,\quad l=0,1,\dots,q,\quad m=2,3,\dots,M,
\end{aligned}
\end{equation}
where
$$
f^m_j(\cdot)=\frac{1}{k_m}\int_{I_m} f(t,\cdot)\psi_j\left(\frac{t-t_{m-1}}{k_m}\right)dt,
$$
and the rational functions
\begin{equation}\label{eq: rational functions r_lj}
r_{l,j}=\frac{p_{l,j}(\lambda)}{\hat{p}(\lambda)}, \quad l,j=0,1,\dots,q,
\end{equation}
are as in the homogenous case with $\hat{p}$ being a polynomial of degree $q+1$ with no roots on the right half complex plane  and $p_{l,j}$, $l,j=0,1,\dots,q$ being polynomials of degree $q$ (cf.~\cite{ErikssonK_JohnsonC_LarssonS_1998a}, page 1322).

Notice that for $m=1,2,\dots,M$,
\begin{equation}\label{eq: estimates for fjm}
\|f^m_j\|_{L^p(\Om)}\le C\|f\|_{L^\infty(I_m; L^p(\Om) )}\quad \text{and}\quad \|f^m_j\|_{L^p(\Om)}\le Ck_m^{-1}\|f\|_{L^1(I_m; L^p(\Om))}.
\end{equation}

\begin{theorem}[Maximal parabolic regularity]\label{thm: maximal parabolic regularity_dGr}
Let $u_k$ satisfy \eqref{eq: dGr nonhomogeneous equation} with $f \in L^s(I;L^p(\Om))$ for $1\le s,p\le \infty$. There exists a constant $C$ independent of $k$ and $f$ such that
$$
\|\Delta u_k\|_{L^s(I;L^p(\Om))}\le C\lk\|f\|_{L^s(I;L^p(\Om))}.
$$
\end{theorem}
\begin{proof}
Using \eqref{eq: dG(r) inhomogeneous one step}, we have the following representation
\begin{equation}\label{eq: nonhomogeneous represnatation dgr}
 U_l^m=k_mG^m_l+r_{l,0}(-k_m\Delta)\sum_{n=1}^{m-1}k_n\left(\prod_{j=1}^{m-n-1}r_{q,0}(-k_{m-j-1}\Delta)\right)G^n_q,
\end{equation}
where
$$
G^m_l = \sum_{j=0}^qr_{l,j}(-k_m\Delta)f_j^m, \quad m=1,2,\dots,M.
$$
with the usual convention that $\prod_{j=1}^0$ is an empty product.
The proof now follows along the lines of Theorem \ref{thm: maximal parabolic regularity}. Taking the Laplacian of both sides we obtain
$$
 \Delta U_l^m=k_m\Delta G^m_l+\Delta r_{l,0}(-k_m\Delta)\sum_{n=1}^{m-1}k_n\left(\prod_{j=1}^{m-n-1}r_{q,0}(-k_{m-j-1}\Delta)\right)G^n_q,
$$
and as a result
$$
 \|\Delta U_l^m\|_{L^p(\Om)}\le \|k_m\Delta G^m_l\|_{L^p(\Om)}
+\left\|\Delta r_{l,0}(-k_m\Delta)\sum_{n=1}^{m-1}k_n\left(\prod_{j=1}^{m-n-1}r_{q,0}(-k_{m-j-1}\Delta)\right)G^n_q\right\|_{L^p(\Om)}.
$$
By Lemma \ref{lem: raional function estimates one more}, we have
\begin{subequations}
\begin{equation}\label{eq: estimates for kmDelta Glm in terms of fjm}
\|k_m\Delta G^m_l\|_{L^p(\Om)}\le C\max_{0\le j\le q}\|f^m_j\|_{L^p(\Om)}, \quad l=0,1\dots,q,
\end{equation}
and by Lemma \ref{lem: raional function estimates} we also have
\begin{equation}\label{eq: estimates for Glm in terms of fjm}
\| G^m_l\|_{L^p(\Om)}\le C\max_{0\le j\le q}\|f^m_j\|_{L^p(\Om)}, \quad l=0,1\dots,q.
\end{equation}
\end{subequations}
On the other hand by Remark \ref{remark: homogenous laplaca dG_r} for any $l=0,1,\dots,q$,
since each term in the sum on the right-hand side can be thought of as a homogeneous solution with initial condition $G_q^n$ at $t=t_{n-1}$, we have
\begin{equation}\label{eq: estimates for sum product G_n}
\left\|\Delta r_{l,0}(-k_m\Delta)\sum_{n=1}^{m-1}k_n\left(\prod_{j=1}^{m-n-1}r_{q,0}(-k_{m-j-1}\Delta)\right)G^n_q\right\|_{L^p(\Om)}\le
C\sum_{n=1}^{m-1}\frac{k_n}{t_m-t_{n-1}}\|G^n_q\|_{L^p(\Om)}.
\end{equation}
To establish the result for $s=\infty$, we observe
\begin{align*}
\|\Delta u_k\|_{L^\infty(I;L^p(\Om))}&=\max_{1\le m\le M}\max_{0\le l\le q}\|\Delta U_l^m\|_{L^p(\Om)}\\
&\le C\max_{1\le m\le M}\max_{0\le j\le q}\|f^m_j\|_{L^p(\Om)}
+C\max_{1\le m\le M}\sum_{n=1}^{m-1}\frac{k_n}{t_m-t_{n-1}}\|G^n_q\|_{L^p(\Om)}\\
&\le C\max_{1\le m\le M}\max_{0\le j\le q}\|f^m_j\|_{L^p(\Om)}\left(1+\max_{1\le m\le M}\sum_{n=1}^{m-1}\frac{k_n}{t_m-t_{n-1}}\right)\\
&\le C\ln{\frac{T}{k}}\max_{1\le m\le M}\max_{0\le j\le q}\|f^m_j\|_{L^p(\Om)},
\end{align*}
where in the last step we used \eqref{eq: estimating sum by integral for log}.
Using \eqref{eq: estimates for fjm} we can conclude that for $s=\infty$
$$
\|\Delta u_k\|_{L^\infty(I;L^p(\Om))}\le C\ln{\frac{T}{k}}\max_{1\le m\le M}\|f\|_{L^\infty(I_m;L^p(\Om))}\le C\ln{\frac{T}{k}}\|f\|_{L^\infty(I;L^p(\Om))}.
$$
Similarly, for $s=1$, we have
\begin{align*}
\|\Delta u_k\|_{L^1(I;L^p(\Om))}&\le \sum_{m=1}^Mk_m\max_{0\le l\le q}\|\Delta U_l^m\|_{L^p(\Om)}\\
&\le C\sum_{m=1}^Mk_m\max_{0\le j\le q}\|f^m_j\|_{L^p(\Om)}
+C\sum_{m=1}^Mk_m\sum_{n=1}^{m-1}\frac{k_n}{t_m-t_{n-1}}\|G^n_q\|_{L^p(\Om)}\\
&\le C\sum_{m=1}^Mk_m\max_{0\le j\le q}\|f^m_j\|_{L^p(\Om)}+C\sum_{m=1}^Mk_m\sum_{n=1}^{m-1}\frac{k_n}{t_m-t_{n-1}}\max_{0\le j\le q}\|f^n_j\|_{L^p(\Om)}\\
&\le C\sum_{m=1}^Mk_m\sum_{n=1}^{m}\frac{k_n}{t_m-t_{n-1}}\max_{0\le j\le q}\|f^n_j\|_{L^p(\Om)}.
\end{align*}
Changing the order of summation and using \eqref{eq: estimating sum by integral for log} we obtain,
\begin{align*}
\sum_{m=1}^Mk_m\sum_{n=1}^{m}\frac{k_n}{t_m-t_{n-1}}\max_{0\le j\le q}\|f^n_j\|_{L^p(\Om)}&\le \sum_{n=1}^Mk_n\max_{0\le j\le q}\|f^n_j\|_{L^p(\Om)}\sum_{m=n}^{M}\frac{k_m}{t_m-t_{n-1}}\\
&\le C\lk \sum_{n=1}^Mk_n\max_{0\le j\le q}\|f^n_j\|_{L^p(\Om)}.
\end{align*}
Thus, by using \eqref{eq: estimates for fjm}, we have
$$
\|\Delta u_k\|_{L^1(I;L^p(\Om))}\le C\ln{\frac{T}{k}}\sum_{m=1}^Mk_m\max_{0\le j\le q}\|f^m_j\|_{L^p(\Om)}\le C\ln{\frac{T}{k}}\|f\|_{L^1(I;L^p(\Om))}.
$$
Interpolating between $s=1$ and $s=\infty$ we obtain the result for any $1\le s\le\infty$.
\end{proof}
\begin{remark}
As in the case of dG($0$) the appearance of a logarithmic term is natural, since in contrast to the continuous case the choices $s,p \in \{1,\infty\}$ are allowed. The power of the logarithm can be improved for $p=2$ or $s=2$. In fact, we can obtain the following estimates,
$$
\|\Delta u_{k}\|_{L^s(I;L^2(\Om))}\le C\left(\lk\right)^{\frac{\abs{s-2}}{s}}\|f\|_{L^s(I;L^2(\Om))},
$$
and
$$
\|\Delta u_{k}\|_{L^2(I;L^p(\Om))}\le C\left(\lk\right)^{\frac{\abs{p-2}}{p}}\|f\|_{L^2(I;L^p(\Om))}.
$$

\end{remark}

\begin{theorem}[Maximal parabolic regularity for jumps]\label{thm: maximal parabolic regularity_dGr jumps}
Let $u_k$ satisfy \eqref{eq: dGr nonhomogeneous equation} with $f \in L^s(I;L^p(\Om))$ for $1\le s,p\le \infty$. Then there exists a constant $C$ independent of $k$ and $f$ such that
\begin{align*}
\max_{1\le m \le M}\left\|\frac{[u_k]_{m-1}}{k_m} \right\|_{L^p(\Om)}&\le C\lk\| f\|_{L^\infty(I;L^p(\Om))}, \quad \text{for } s= \infty,\\
\left(\sum_{m=1}^Mk_m\left\|\frac{[u_k]_{m-1}}{k_m} \right\|^s_{L^p(\Om)}\right)^{\frac{1}{s}}&\le C\lk\| f\|_{L^s(I;L^p(\Om))}, \quad \text{for } 1\le s< \infty.
\end{align*}
\end{theorem}
\begin{proof}
Using \eqref{eq: dG(r) inhomogeneous one step} and \eqref{eq: nonhomogeneous represnatation dgr}, we have the following representation for the jump terms
\begin{align*}
\frac{[u_k]_{m-1}}{k_m}&=\frac{U^m_0-U^{m-1}_q}{k_m}\\
&=G^m_0+k_m^{-1}\left(r_{0,0}(-k_m\Delta)U^{m-1}_q-U^{m-1}_q\right)=G^m_0+k_m^{-1}\left(r_{0,0}(-k_m\Delta)-\operatorname{Id}\right)U^{m-1}_q.
\end{align*}
Using  that $r_{0,0}-1$ satisfies \eqref{eq: proprties of r_l0} and using Lemma \ref{lemma: rational with z}, Lemma \ref{lem: raional function estimates},  and  proceeding similarly to the proof of Theorem \ref{thm: maximal parabolic regularity_dGr}, we have
\begin{equation}\label{eq: after apllying all inequalities}
\begin{aligned}
k_m^{-1}\|[u_k]_{m-1}\|_{L^p(\Om)}&\le C\left(\|G^m_0\|_{L^p(\Om)}+\|\Delta U^{m-1}_q\|_{L^p(\Om)}\right)\\
&\le C\max_{0\le j\le q}\|f^m_j\|_{L^p(\Om)}+C\sum_{n=1}^{m-1}\frac{k_n}{t_m-t_{n-1}}\max_{0\le j\le q}\|f^n_j\|_{L^p(\Om)}\\
&\le C\sum_{n=1}^{m}\frac{k_n}{t_m-t_{n-1}}\max_{0\le j\le q}\|f^n_j\|_{L^p(\Om)}.
\end{aligned}
\end{equation}
Now, the proof of the cases $s=1$ and $s=\infty$ is identical to the one of the  previous Theorem \ref{thm: maximal parabolic regularity_dGr} and we have
\begin{align*}
\max_{1\le m \le M}\left\|\frac{[u_k]_{m-1}}{k_m} \right\|_{L^p(\Om)}&\le C\lk\| f\|_{L^\infty(I;L^p(\Om))}, \quad 1\le p\le \infty,\\
\sum_{m=1}^Mk_m\left\|\frac{[u_k]_{m-1}}{k_m} \right\|_{L^p(\Om)}&\le C\lk\| f\|_{L^1(I;L^p(\Om))}, \quad 1\le p\le \infty.
\end{align*}
For $1<s<\infty$ using the H\"{o}lder inequality with $\frac{1}{s}+\frac{1}{s'}=1$, we obtain,
\begin{equation}\label{eq: after Holder}
\begin{aligned}
\left\|\frac{[u_k]_{m-1}}{k_m} \right\|_{L^p(\Om)}&\le C\sum_{n=1}^{m}\frac{k_n}{t_m-t_{n-1}}\max_{0\le j\le q}\|f^n_j\|_{L^p(\Om)}\\
&\le C\left(\sum_{n=1}^m  \frac{k_n}{t_m-t_{n-1}}\max_{0\le j\le q}\|f^n_j\|^s_{L^p(\Om)}\right)^{1/s}\left(\sum_{n=1}^m \frac{k_n}{t_m-t_{n-1}}\right)^{1/s'}\\
&\le C\left(\lk\right)^{1/s'}\left(\sum_{n=1}^m \frac{k_n}{t_m-t_{n-1}}\max_{0\le j\le q}\|f^n_j\|^s_{L^p(\Om)}\right)^{1/s}.
\end{aligned}
\end{equation}
Hence
$$
\sum_{m=1}^Mk_m\left\|\frac{[u_k]_{m-1}}{k_m} \right\|^s_{L^p(\Om)}\le C\left(\lk\right)^{s/s'}\sum_{m=1}^Mk_m\sum_{n=1}^m \frac{k_n}{t_m-t_{n-1}}\max_{0\le j\le q}\|f^n_j\|^s_{L^p(\Om)}.
$$
Changing the order of summation, we obtain
$$
\begin{aligned}
\sum_{m=1}^Mk_m\left\|\frac{[u_k]_{m-1}}{k_m} \right\|^s_{L^p(\Om)}&\le C\left(\lk\right)^{s/s'}\sum_{n=1}^M k_n\max_{0\le j\le q}\|f^n_j\|^s_{L^p(\Om)} \sum_{m=n}^M \frac{k_m}{t_m-t_{n-1}}\\&\le C\left(\lk\right)^{1+s/s'}\sum_{n=1}^M k_n\max_{0\le j\le q}\|f^n_j\|^s_{L^p(\Om)}=C\left(\lk\right)^{s}\|f\|^s_{L^s(I;L^p(\Om))}.
\end{aligned}
$$
Taking the $s$-root we finish the proof.
\end{proof}

\begin{theorem}\label{thm: time derivative maximal parabolic dGq}
Let $u_k$ satisfy \eqref{eq: dGr nonhomogeneous equation}. Then there exists a constant $C$ independent of $k$ and $f$ such that
$$
\left(\sum_{m=1}^M\|\pa_t u_k\|^s_{L^s(I_m;L^p(\Om))}\right)^{\frac{1}{s}}\le C\ln{\frac{T}{k}}\|f\|_{L^s(I;L^p(\Om))}, \quad 1\le s< \infty, \quad 1\le p\le \infty.
$$
\end{theorem}
\begin{proof}
Similarly to the proof of Theorem \ref{thm: homogeneous smoothing dG_r jumps},
using \eqref{eq: formular of dGq of u_k on I_m} and \eqref{eq: dG(r) inhomogeneous one step}, we have
$$
\begin{aligned}
\pa_t u_k|_{I_m}&=k_m^{-1}\sum_{l=0}^q U^m_{l}(x)\psi_l'\left(\frac{t-t_{m-1}}{k_m}\right)+\sum_{l=0}^q G^m_{l}(x)\psi_l'\left(\frac{t-t_{m-1}}{k_m}\right)\\
&=k_m^{-1}\sum_{l=0}^q r_{l,0}(-k_m\Delta) \psi_l'\left(\frac{t-t_{m-1}}{k_m}\right)U^{m-1}_{q}(x)+\sum_{l=0}^q G^m_{l}(x)\psi_l'\left(\frac{t-t_{m-1}}{k_m}\right).
\end{aligned}
$$
Using \eqref{eq: proprties of r_l0} and  $\sum_{l=0}^q \psi_l'\left(\frac{t-t_{m-1}}{k_m}\right)=0$, we can conclude that
$$
\sum_{l=0}^q r_{l,0}(z) \psi_l'\left(\frac{t-t_{m-1}}{k_m}\right)=\frac{z\tilde{p}_t(z)}{\hat{p}(z)},
$$
where $\hat{p}(z)$ is the same polynomial as in \eqref{eq: rational function rlj} and $\tilde{p}_t(z)$ is some polynomial of degree $q$ whose coefficients are time dependent, but uniformly bounded on $I_m$.
Thus again by Lemma \ref{lemma: rational with z} and Lemma \ref{lem: raional function estimates one more},
we obtain
$$
\|\pa_t u_k\|_{L^\infty(I_m;L^p(\Om))}\le C\|\Delta U^{m-1}_{q}\|_{L^p(\Om)}+C\max_{0\le j\le q}\|f^m_j\|_{L^p(\Om)}.
$$
The rest of the proof is identical to the proof of the previous theorem.
\end{proof}

\subsection{Application to optimal order error estimates.}

As an application of the maximal parabolic regularity, we show optimal convergence rates for the dG($q$) solution. First, we establish that the error is bounded by a certain projection error. A similar result was obtained for the ${L^2(I;L^2(\Om))}$ norm in \cite{DMeidner_BVexler_2007a}. First, we define a projection $\pi_k$ for $u \in C(I,L^2(\Omega))$ with $\pi_k u|_{I_m} \in P_q(L^2(\Omega))$ for $m=1,2,\dots,M$ on each subinterval $I_m$ by
\begin{subequations}\label{eq: projection pi_k}
\begin{align}
(\pi_k u-u,\phi)_{I_m\times \Omega}&=0,\quad \forall \phi\in P_{q-1}(I_m,L^2(\Omega)),\quad q>0,\label{eq: projection pi_k first}\\
\pi_k u(t_m^-)=u(t_m^-) \label{eq: projection pi_k second}.
\end{align}
\end{subequations}
In the case $q = 0$, $\pi_ku$ is defined solely by the second condition.
\begin{theorem}\label{th:semi_discrete_error_proj}
Let $u$ be the solution to \eqref{eq: heat equation} with $u \in C(\bar I; L^p(\Omega))$ and $u_k$ be its dG($q$) approximation~\eqref{eq: semidiscrete heat with RHS}, for $q\geq 0$. Then there exists a constant $C$ independent of $k$ such that
$$
\|u-u_k\|_{L^s(I;L^p(\Om))}\le C\lk \|u-\pi_k u\|_{L^s(I;L^p(\Om))},\quad 1\le s,p<\infty,
$$
where the projection $\pi_k$ is defined above in \eqref{eq: projection pi_k}.
\end{theorem}
\begin{proof}
Put $e:=u-u_k=(u-\pi_k u)+(\pi_ku-u_k):=\eta_k+\xi_k$.
For $1\le s,p<\infty$, we have
$$
\|e\|_{L^s(I;L^p(\Om))}=\sup_{\psi\in L^{s'}(I;L^{p'}(\Om))\atop{\|\psi\|_{L^{s'}(I;L^{p'}(\Om))}=1}}(e,\psi)_{I\times\Om}, \quad \frac{1}{s}+\frac{1}{s'}=1,\quad \frac{1}{p}+\frac{1}{p'}=1.
$$
For each such $\psi$, we consider a dual problem for $z_k \in X_k^q$ satisfying
$$
B(\varphi_k, z_k)=(\varphi_k, \psi)_{I\times\Om} \quad \text{for all }\; \varphi_k \in X_k^q.
$$
Thus, we have
$$
(e,\psi)_{I\times\Om}=(\eta_k,\psi)_{I\times\Om}+(\xi_k,\psi)_{I\times\Om}:=J_1+J_2.
$$
Using the H\"{o}lder inequality, we find
$$
J_1\le \|\eta_k\|_{L^s(I;L^p(\Om))}\|\psi\|_{L^{s'}(I;L^{p'}(\Om))}\le \|\eta_k\|_{L^s(I;L^p(\Om))}.
$$
On the other hand using that $B(u-u_k,\chi_k)=0$ for any $\chi_k\in X_k^q$,  we obtain
\begin{align*}
J_2=B(\xi_k, z_k)=-B(\eta_k, z_k)&=\sum_{m=1}^M (\eta_k,\pa_t z_k)_{I_m\times\Om}-(\na\eta_k,\na z_k)_{I_m\times\Om}+(\eta_{k,m}^-,[z_k]_m)_\Om\\&=-(\na\eta_k,\na z_k)_{I\times\Om},
\end{align*}
where we used that the first sum vanishes due to \eqref{eq: projection pi_k first} and the sum involving jumps due to \eqref{eq: projection pi_k second}. Integrating by parts in space, using the H\"{o}lder inequality and Theorem \ref{thm: maximal parabolic regularity_dGr}, we obtain
\begin{align*}
J_2=-(\na\eta_k,\na z_k)_{I\times\Om}&=(\eta_k,\Delta z_k)_{I\times\Om}\le \|\eta_k\|_{L^s(I;L^p(\Om))}\|\Delta z_k\|_{L^{s'}(I;L^{p'}(\Om))}\\
&\le C\lk\|\eta_k\|_{L^s(I;L^p(\Om))}\|\psi\|_{L^{s'}(I;L^{p'}(\Om))}\le C\lk\|\eta_k\|_{L^s(I;L^p(\Om))}.
\end{align*}
Combining the estimates for $J_1$ and $J_2$ we obtain the result.
\end{proof}

If the exact solution is sufficiently smooth then the above result easily leads to an optimal convergence rate, modulo a logarithmic term.
\begin{corollary}
Let $u\in W^{q+1,s}(I;L^p(\Om))$  be the solution to \eqref{eq: heat equation} and $u_k$ be its dG($q$) approximation for $q\geq 0$. Then there exists a constant $C$ independent of $k$ such that
$$
\|u-u_k\|_{L^s(I;L^p(\Om))}\le Ck^{q+1}\lk \|u\|_{W^{q+1,s}(I;L^p(\Om))},\quad 1\le s,p<\infty.
$$
\end{corollary}
\begin{remark}
The above result can be extended to the case of non-homogeneous Dirichlet boundary conditions. Let $g \in C(I;L^2(\Omega))\cap L^2(I;H^1(\Omega))$ be given and consider the equation
\begin{equation*}
\begin{aligned}
\pa_tu(t,x)-\Delta u(t,x) &= f(t,x), & (t,x) &\in \IOm,\;  \\
    u(t,x) &= g(t,x),    & (t,x) &\in I\times\pa\Omega, \\
   u(0,x) &= u_0(x),    & x &\in \Omega.
\end{aligned}
\end{equation*}
It turns out, that it is convenient to use $\pi_k g$ as boundary conditions for the semidiscrete solution, i.e.
\[
u_k \in \pi_k g +X_k^q \quad:\quad B(u_k,\varphi_k)=(f,\varphi_k)_{\IOm}+(u_0,\varphi_{k,0}^+)_\Om \quad \text{for all }\; \varphi_k\in X_k^q.
\]
Then following the lines of the proof of Theorem~\ref{th:semi_discrete_error_proj} and using that $\xi_k = \pi_k u - u_k$ has homogeneous boundary conditions, i.e., $\xi_k \in X_k^q$, we obtain
\[
(\xi_k,\psi)_{I \times \Omega} = - (\nabla \eta_k, \nabla z_k) = (\eta_k, \Delta z_k)_{I \times \Omega} + \int_I \int_{\partial \Omega} (g - \pi_k g) \partial_n z_k \, ds\, dt.
\]
Under an additional assumption on $\Omega$ that for any $v \in H^1_0(\Omega)$ with $\Delta v \in L^{p'}(\Omega)$ the estimate
\[
\norm{\partial_n v}_{L^{p'}(\partial \Omega)} \le c \norm{\Delta v}_{L^{p'}(\Omega)}
\]
holds, we obtain
\[
\|u-u_k\|_{L^s(I;L^p(\Om))}\le C\lk \left(\|u-\pi_k u\|_{L^s(I;L^p(\Om))} + \norm{g-\pi_k g}_{L^s(I;L^p(\partial \Om))}\right),\quad 1\le s,p<\infty.
\]
The above assumption is fulfilled, for example, if on $\Omega$ the $W^{2,p'}$ elliptic regularity holds.
\end{remark}

\section{Fully discrete solutions}\label{sec: fully discrete}

In this section, we consider the fully discrete approximation of the equation \eqref{eq: heat equation}.
 From now on we assume that the  domain $\Om$ is a polygonal/polyhedral convex domain. For $h \in (0, h_0]$; $h_0 > 0$, let $\mathcal{T}$  denote  a quasi-uniform triangulation of $\Om$  with mesh size $h$, i.e., $\mathcal{T} = \{\tau\}$ is a partition of $\Om$ into cells (triangles or tetrahedrons) $\tau$ of diameter $h_\tau$ such that for $h=\max_{\tau} h_\tau$,
$$
\operatorname{diam}(\tau)\le h \le C |\tau|^{\frac{1}{d}}, \quad \forall \tau\in \mathcal{T},\quad d=2,3,
$$
hold. Let $V_h$ be the set of all functions in $H^1_0(\Om)$ that are polynomials of degree $r$ on each $\tau$, i.e., $V_h$ is the usual space of conforming finite elements.
To obtain the fully discrete approximation we consider the space-time finite element space
\begin{equation} \label{def: space_time}
\Xkh=\{v_{kh} :\ v_{kh}|_{I_m}\in \Ppol{q}(V_h), \ m=1,2,\dots,M, \quad q\geq 0,\quad r\geq 1 \}.
\end{equation}
We define a fully discrete analog $u_{kh} \in \Xkh$ of $u_k$ introduced in  \eqref{eq: semidiscrete heat with RHS} by
\begin{equation}\label{eq: semi fully discrete heat with RHS}
B(u_{kh},\varphi_{kh})=(f,\varphi_{kh})_{\IOm}+(u_0,\varphi_{kh}^+)_\Om \quad \text{for all }\; \varphi_{kh}\in \Xkh.
\end{equation}
Moreover, we introduce the discrete Laplace operator $\Delta_h \colon V_h \to V_h$ by
\[
(-\Delta_h v_h,\chi)_{\Om} = (\nabla v_h,\nabla \chi)_{\Om}, \quad \forall \chi\in V_h.
\]
The semidiscrete results from the first part of the paper translate almost immediately to the fully discrete setting provided we have the corresponding resolvent estimate,
\begin{equation}\label{eq: discrete resolvent}
\|(z+\Delta_h)^{-1}\chi\|_{L^p(\Om)}\le  \frac{C}{1+|z|}\|\chi\|_{L^p(\Om)},\quad {\forall z\in \mathbb{C}\backslash \Sigma_{\gamma}}, \quad\forall \chi\in V_h, \quad 1\le p\le \infty,
\end{equation}
with some constant $C$ independent of $h$.
Such a result  was established in \cite{BakaevNY_ThomeeV_WahlbinLB_2003a} for smooth domains. Later it was extended to convex polyhedral domains in \cite{LiB_SunW_2015a} (for some $\gamma>0$)  via stability  and smoothing properties of the  semigroup $E_h(t)=e^{-\Delta_ht}$ and directly for an arbitrary $\gamma>0$ but with logarithmic dependence of the constant $C$ on $h$  in \cite{LeykekhmanD_VexlerB_2015c}.

\subsection{Result for the homogeneous problem}
\medskip
Let $u_{kh}\in \Xkh$ be the fully discrete dG($q$)cG($r$) solution to the parabolic equation with $f\equiv 0$, i.e.
\begin{equation}\label{eq: dGr homogeneous equation fully}
B(u_{kh},\varphi_{kh})=(u_0,\varphi^+_{kh,0}),\quad \forall \varphi_{kh}\in \Xkh.
\end{equation}
\begin{theorem}[Fully discrete homogeneous smoothing estimate]\label{thm: homogeneous smoothing dG_r fully discrete}
Let $u_{kh}$ be a solution of \eqref{eq: dGr homogeneous equation fully} with $u_0 \in L^p(\Omega)$, $1\le p\le \infty$. Then there exists a constant $C$ independent of $k$ and $h$ such that
$$
\|\pa_tu_{kh} \|_{L^\infty(I_m; L^p(\Om))}+\|\Delta_h u_{kh}\|_{L^\infty(I_m; L^p(\Om))}+k_m^{-1}\|[u_{kh}]_{m-1}\|_{L^p(\Om)}\le \frac{C}{t_m}\|u_0\|_{L^p(\Om)},
$$
for $m=1,2,\dots,M$.
\end{theorem}

\subsection{Results for the inhomogeneous problem}
Let $u_{kh}\in \Xkh$ be the dG($q$)cG($r$) solution to the inhomogeneous parabolic equation with $u_0=0$, i.e.
\begin{equation}\label{eq: dGr nonhomogeneous equation fully}
B(u_{kh},\varphi_{kh})=(f,\varphi_{kh}),\quad \forall \varphi_{kh}\in \Xkh.
\end{equation}
\begin{theorem}[Fully discrete maximal parabolic regularity]\label{thm: maximal parabolic regularity_dGr fully discrete}
Let $u_{kh}$ satisfy \eqref{eq: dGr nonhomogeneous equation fully} with $f \in L^s(I;L^p(\Omega))$, $1\le s,p\le \infty$. Then there exists a constant $C$ independent of $k$ and $h$ such that
$$
\left(\sum_{m=1}^M\|\pa_t u_{kh}\|^s_{L^s(I_m;L^p(\Om))}\right)^{\frac{1}{s}}+\|\Delta_h u_{kh}\|_{L^s(I;L^p(\Om))}+\left(\sum_{m=1}^Mk_m\left\|\frac{[u_{kh}]_{m-1}}{k_m} \right\|^s_{L^p(\Om)}\right)^{\frac{1}{s}}\le C\lk\|f\|_{L^s(I;L^p(\Om))},
$$
with obvious notation changes in the case of $s=\infty.$
\end{theorem}

\subsection{Application to optimal order error estimates.}

Similarly to the semidiscrete case, as an application of the maximal parabolic regularity, we show optimal convergence rates for the dG($q$)cG($r$) solution.
\begin{theorem}
Let $u$  be the solution to \eqref{eq: heat equation} with $u \in C(\bar I; L^p(\Omega))$ and $u_{kh}$ be the dG($q$)cG($r$) solution for $q\geq 0$ and $r\geq 1$. Then there exists a constant $C$ independent of $k$ and $h$ such that for $1\le s,p<\infty$,
$$
\|u-u_{kh}\|_{L^s(I;L^p(\Om))}\le C\lk \left(\|u-\pi_k u\|_{L^{s}(I;L^{p}(\Om))}+\|P_hu-u\|_{L^{s}(I;L^{p}(\Om))}+\|R_hu-u\|_{L^{s}(I;L^{p}(\Om))}\right),
$$
where the projection $\pi_k$ is defined in \eqref{eq: projection pi_k}, $P_h \colon L^2(\Om)\to V_h$ is the orthogonal $L^2$ projection and $R_h \colon H^1_0(\Omega)\to V_h$ is the Ritz projection.
\end{theorem}
\begin{proof}
Put $e:=u-u_{kh}=(u-P_h\pi_k u)+(P_h\pi_ku-u_{kh}):=\eta_{kh}+\xi_{kh}$.
For $1\le s,p<\infty$, we have
$$
\|e\|_{L^s(I;L^p(\Om))}=\sup_{\psi\in L^{s'}(I;L^{p'}(\Om))\atop{\|\psi\|_{L^{s'}(I;L^{p'}(\Om))}=1}}(e,\psi)_{I\times\Om}, \quad \frac{1}{s}+\frac{1}{s'}=1,\quad \frac{1}{p}+\frac{1}{p'}=1.
$$
For each such $\psi$, consider a dual problem
$$
B(\varphi_{kh}, z_{kh})=(\varphi_{kh}, \psi)_{I\times\Om}.
$$
Thus, we have
$$
(e,\psi)_{I\times\Om}=(\eta_{kh},\psi)_{I\times\Om}+(\xi_{kh},\psi)_{I\times\Om}:=J_1+J_2.
$$
Using the H\"{o}lder inequality, the triangle inequality, the stability of the $L^2$ projection $P_h$ in $L^p(\Om)$ and the approximation properties of $\pi_k$ and $P_h$, we find
\begin{align*}
J_1&\le C\|\eta_{kh}\|_{L^s(I;L^p(\Om))}\|\psi\|_{L^{s'}(I;L^{p'}(\Om))}\le C\|\eta_{kh}\|_{L^s(I;L^p(\Om))}
=C\|u-P_h\pi_k u\|_{L^s(I;L^p(\Om))}\\
&\le C\|u-P_h u\|_{L^s(I;L^p(\Om))}+C\|P_h(u-\pi_k u)\|_{L^s(I;L^p(\Om))}\\
&\le  C\|u-P_h u\|_{L^s(I;L^p(\Om))}+C\|u-\pi_k u\|_{L^s(I;L^p(\Om))}.
\end{align*}
On the other hand, using that $B(u-u_{kh},\chi_{kh})=0$ for any $\chi_{kh}\in \Xkh$, and the properties of the $L^2$ projection and the properties of $\pi_k$,  we obtain
\begin{align*}
J_2&=B(\xi_{kh}, z_{kh})=-B(\eta_{kh}, z_{kh})=\sum_{m=1}^M (\eta_{kh},\pa_t z_{kh})_{I_m\times\Om}-(\na\eta_{kh},\na z_{kh})_{I_m\times\Om}+(\eta_{kh,m}^-,[z_{kh}]_m)_\Om\\
&=\sum_{m=1}^M (u-\pi_k u,\pa_t z_{kh})_{I_m\times\Om}-(\na\eta_{kh},\na z_{kh})_{I_m\times\Om}+(u_m^--(\pi_k u)_m^-,[z_{kh}]_m)_\Om\\
&=-(\na(u-P_h \pi_k u),\na z_{kh})_{I\times\Om}.
\end{align*}
where we used that the first sum vanishes due to \eqref{eq: projection pi_k first} and the sum involving jumps due to \eqref{eq: projection pi_k second}. Using the properties of the Ritz projection, integrating by parts in space, and using the H\"{o}lder inequality and Theorem \ref{thm: maximal parabolic regularity_dGr}, we obtain
\begin{align*}
J_2=-(\na(u-P_h \pi_k u),\na z_{kh})_{I\times\Om}&=-(\na(R_hu-P_h \pi_k u),\na z_{kh})_{I\times\Om}=(R_hu-P_h \pi_k u,\Delta_h z_{kh})_{I\times\Om}\\
&\le C\|P_h(R_hu-\pi_k u)\|_{L^s(I;L^p(\Om))}\|\Delta_h z_{kh}\|_{L^{s'}(I;L^{p'}(\Om))}\\
&\le C\lk\|R_hu-\pi_k u\|_{L^s(I;L^p(\Om))}\|\psi\|_{L^{s'}(I;L^{p'}(\Om))}\\
&\le C\lk\left(\|R_hu- u\|_{L^s(I;L^p(\Om))}+\|u-\pi_k u\|_{L^s(I;L^p(\Om))}\right).
\end{align*}
Combining the estimates for $J_1$ and $J_2$ we obtain the result.
\end{proof}
\begin{corollary}
If the solution $u$ to \eqref{eq: heat equation} satisfies $u\in W^{q+1,s}(I;L^p(\Om))\cap L^{s}(I;W^{r+1,p}(\Om))$  and $\Omega$ such that elliptic $W^{2,p'}$- regularity holds, then there exists a constant $C$ independent of $k$ and $h$ such that
$$
\|u-u_{kh}\|_{L^s(I;L^p(\Om))}\le C\lk \left(k^{q+1}\|u\|_{W^{q+1,s}(I;L^p(\Om))}+h^{r+1}\|u\|_{L^{s}(I;W^{r+1,p}(\Om))}\right),\quad 1\le s,p<\infty.
$$
\end{corollary}
\section{Fully discrete results in general norms}\label{sec:general_norms}

For the future references we provide discrete maximal parabolic regularity results in general norms. For example, we use these results to establish pointwise best approximation estimates in~\cite{LeykekhmanD_VexlerB_2015d} for fully discrete Galerkin solutions.

Let $\Om$ be a Lipschitz domain and let $\mathcal{T} = \{\tau\}$ be an arbitrary partition of $\Om$ into cells $\tau$ (triangles, tetrahedrons, quads, or hexahedrons, not necessary quasi-uniform). Let $V_h$ be the set of all functions in $H^1_0(\Om)$ that belong to a certain polynomial space (i.e., $P_r$ or $Q_r$) on each $\tau$.
As before, we define a fully discrete solution $u_{kh} \in \Xkh$  by
\begin{equation}\label{eq: semi fully discrete heat with RHS general}
B(u_{kh},\varphi_{kh})=(f,\varphi_{kh})_{\IOm}+(u_0,\varphi_{kh}^+)_\Om \quad \text{for all }\; \varphi_{kh}\in \Xkh,
\end{equation}
where
\begin{equation} \label{def: space_time general}
\Xkh=\{v_{kh} :\ v_{kh}|_{I_m}\in \Ppol{q}(V_h), \ m=1,2,\dots,M\}, \quad \text{for some }\; q\geq 0,\quad r\geq 1.
\end{equation}
As in the previous section, we introduce the discrete Laplace operator $\Delta_h \colon V_h \to V_h$ by
\[
(-\Delta_h v_h,\chi)_{\Om} = (\nabla v_h,\nabla \chi)_{\Om}, \quad \forall \chi\in V_h,
\]
and the orthogonal $L^2$ projection $P_h \colon L^2(\Om) \to V_h$ by
\[
(P_h v,\chi)_{\Om} = ( v, \chi)_{\Om}, \quad \forall \chi\in V_h.
\]

Let $\vertiii{\cdot}$ be a norm on $V_h$ such that for some $\gamma\in(0,\frac{\pi}{2})$ the following resolvent estimate holds,
\begin{equation}\label{eq: resolvent in Banach space}
\vertiii{(z+\Delta_h)^{-1}\chi} \le \frac{M_h}{\abs{z}} \vertiii{\chi},\quad\text{for}\ z\in \mathbb{C}\setminus \Sigma_{\gamma},
\end{equation}
for all $\chi \in V_h$, where $\Sigma_{\gamma}$ is defined in \eqref{eq: definition of sigma} and
 the constant $M_h$ is independent of $z$.

For quasi-uniform meshes, this assumption is fulfilled for $\vertiii{\cdot}=\norm{\cdot}_{L^p(\Omega)}$ with a constant $M_h \le C$ independent of $h$, see~\cite{LiB_SunW_2015a}, as discussed and exploited above.
For a weighted norm $\vertiii{\cdot}=\norm{\sigma^{\frac{N}{2}}\cdot}_{L^2(\Omega)}$ with the weight $\sigma_{x_0}(x)=\sqrt{|x-x_0|^2+h^2}$ and $M_h \le C \lh$ we established this estimate in~\cite{LeykekhmanD_VexlerB_2015d}, and used the corresponding result to obtain interior (local) pointwise estimates. Moreover, the resolvent estimate \eqref{eq: resolvent in Banach space} is known also to hold in ${L^p(\Omega)}$ norms on a class of non quasi-uniform meshes as well, see~\cite{BakaevNY_CrouzeixM_ThomeeV_2006a}.

\subsection{Smoothing estimates for the homogeneous problem in general norms}

For the homogeneous heat equation~\eqref{eq: heat equation}, i.e.  $f=0$ and its discrete approximation $u_{kh} \in \Xkh$ defined by
\begin{equation}\label{eq: dg(r) homogeneous general}
B(u_{kh},\varphi_{kh}) = (u_0,\varphi_{kh,0}^+) \quad \forall \varphi_{kh} \in \Xkh,
\end{equation}
we have the following  smoothing result.
\begin{theorem}[Fully discrete smoothing estimate in general norms]\label{lemma: homogeneous smoothing dG_r fully discrete general}
Let $\vertiii{\cdot}$ be a norm on $V_h$ fulfilling the resolvent estimate~\eqref{eq: resolvent in Banach space}. Let $u_{kh}$ be the solution  of \eqref{eq: dg(r) homogeneous general}. Then, there exists a constant $C$ independent of $k$ and $h$ such that
$$
\sup_{t\in I_m}\vertiii{\pa_t u_{kh}(t)}+\sup_{t\in I_m}\vertiii{\Delta_h u_{kh}(t)}+k_m^{-1}\vertiii{[u_{kh}]_{m-1}}\le \frac{CM_h}{t_m}\vertiii{P_h u_0},
$$
for $m=1,2,\dots,M$, where $P_h \colon L^2(\Om)\to V_h$ is the orthogonal $L^2$ projection. For $m=1$ the jump term is understood as $[u_{kh}]_0 = u_{kh,0}^+-P_h u_0$.
\end{theorem}

\subsection{Discrete maximal parabolic estimates for the inhomogeneous problem in general norms}

Now, we consider the inhomogeneous heat equation \eqref{eq: heat equation},  with $u_0=0$ and its discrete approximation $u_{kh} \in \Xkh$ defined by
\begin{equation}\label{eq: dGr nonhomogeneous equation fully2}
B(u_{kh},\varphi_{kh})=(f,\varphi_{kh}),\quad \forall \varphi_{kh}\in \Xkh.
\end{equation}
\begin{theorem}[Discrete maximal parabolic regularity in general norms]\label{lemma: fully discrete_maximal_parabolic general}
Let $\vertiii{\cdot}$ be a norm on $V_h$ fulfilling the resolvent estimate~\eqref{eq: resolvent in Banach space} and let $1 \le s \le \infty$. Let $u_{kh}$ be a solution  of \eqref{eq: dGr nonhomogeneous equation fully2}. Then, there exists a constant $C$ independent of $k$ and $h$ such that
\begin{equation*}
\begin{aligned}
\left(\sum_{m=1}^M\int_{I_m}\vertiii{\pa_t u_{kh}(t)}^sdt\right)^{\frac{1}{s}}+\left(\sum_{m=1}^M\int_{I_m}\vertiii{\Delta_h u_{kh}(t)}^sdt\right)^{\frac{1}{s}}&+\left(\sum_{m=1}^Mk_m\vertiii{k_m^{-1}[u_{kh}]_{m-1}}^s\right)^{\frac{1}{s}}\\
&\le C M_h \lk\left(\int_I\vertiii{P_h f(t)}^sdt\right)^{\frac{1}{s}},
\end{aligned}
\end{equation*}
where $P_h \colon L^2(\Om)\to V_h$ is the orthogonal $L^2$ projection and with obvious notation change in the case of $s=\infty$. For $m=1$ the jump term is understood as $[u_{kh}]_0=u_{kh,0}^+$.
\end{theorem}

The proofs of the above two results are identical to the proofs of the corresponding  time discrete results  from Section \ref{sec: dGq}, provided  the resolvent estimate \eqref{eq: resolvent in Banach space} holds.

\begin{acknowledgements}
The authors would like to thank Dominik Meidner and Konstantin Pieper for the careful reading of the manuscript and providing valuable suggestions that help to improve the presentation of the paper.
\end{acknowledgements}

\bibliography{litMaxParReg}
\bibliographystyle{spmpsci}      

%
%


\end{document}